\documentclass{article}
\usepackage{amsfonts}
\usepackage{amsmath}
\usepackage{pstricks}
\usepackage{graphicx}

\newtheorem{theorem}{Theorem}
\newtheorem{corollary}[theorem]{Corollary}
\newtheorem{definition}[theorem]{Definition}

\newtheorem{lemma}[theorem]{Lemma}
\newtheorem{proposition}[theorem]{Proposition}

\newenvironment{proof}[1][Proof]{\textbf{#1.} }{\ \rule{0.5em}{0.5em}}
\title{The Uniform Order Convergence Structure on $\mathcal{ML}\left(X\right)$}

\author{J H van der Walt\\ \\
Department of Mathematics and Applied Mathematics\\ University of
Pretoria\\ Pretoria 0002\\ South Africa}

\begin{document}

\maketitle

\begin{abstract}
The aim of this paper is to set up appropriate uniform convergence
spaces in which to reformulate and enrich the Order Completion
Method \cite{Obergugenberger and Rosinger} for nonlinear PDEs.  In
this regard, we consider an appropriate space
$\mathcal{ML}\left(X\right)$ of normal lower semi-continuous
functions.  The space $\mathcal{ML}\left(X\right)$ appears in the
ring theory of $\mathcal{C}\left(X\right)$ and its various
extensions \cite{Fine Gillman and Lambek}, as well as in the
theory of nonlinear, PDEs \cite{Obergugenberger and Rosinger} and
\cite{Rosinger 3}.  We define a uniform convergence structure, in
the sense of \cite{Beattie and Butzmann}, on
$\mathcal{ML}\left(X\right)$ such that the induced convergence
structure is the order convergence structure, as introduced in
\cite{Anguelov and van der Walt} and \cite{van der Walt 1}.  The
uniform convergence space completion of
$\mathcal{ML}\left(X\right)$ is constructed as the space all
normal lower semi-continuous functions on $X$. It is then shown
how these ideas may be applied to solve nonlinear PDEs. In
particular, we construct generalized solutions to the
Navier-Stokes Equations in three spatial dimensions, subject to an
initial condition.
\end{abstract}
\mbox{ }\\
Keywords  General Topology,  Uniform Convergence Structures,
 Function
Spaces,  Ordered Spaces \\
\\
2000 Mathematics Subject Classification 54A20, 46E05, 06F30

\section{Introduction}

It is widely held that, in contradistinction to ODEs, there can be
no general, type independent theory for the existence and
regularity of the solutions to PDEs \cite{Arnold}, \cite{Evans}.
As seen in the sequel, this is in fact a misunderstanding which is
often attributed to the more complex geometry of $\mathbb{R}^{n}$,
with $n\geq 2$, as apposed to that of $\mathbb{R}$ which is
relevant to ODEs alone, see \cite{Arnold}.  Indeed, the
difficulties that are typically encountered when solving PDEs by
the usual function analytic methods, which are perceived to arise
form the complicated geometry of $\mathbb{R}^{n}$, are rather due
to the inherent limitations of the function analytic methods
themselves, and are therefore technical obstacles, rather than
conceptual ones.

The above is exemplified by the appearance of not only one, but
two general, type independent theories for the solutions of
nonlinear PDEs. The Central Theory of PDEs, developed by Neuberger
\cite{Neuberger 1}, see also \cite{Neuberger 4}, is based on a
generalized method of steepest descent in suitably constructed
Hilbert Spaces.  The Order Completion Method, as developed by
Oberguggenberger and Rosinger \cite{Obergugenberger and Rosinger},
is based on the Dedekind completion of suitable spaces of
equivalence classes of functions.

\subsection{The Order Completion Method}

The method of Order Completion results in the existence of
generalized solutions to arbitrary, continuous nonlinear PDEs of
the form
\begin{eqnarray}
T\left(x,D\right)u\left(x\right)=f\left(x\right)\mbox{,
}x\in\Omega\label{PDE}
\end{eqnarray}
with the right hand term $f$ a continuous function of
$x\in\Omega$, and the partial differential operator
$T\left(x,D\right)$ defined through a jointly continuous function
\begin{eqnarray}
F:\Omega\times\mathbb{R}^{M}\rightarrow \mathbb{R}\nonumber
\end{eqnarray}
by
\begin{eqnarray}
T\left(x,D\right)u:x\mapsto
F\left(x,u\left(x\right),...,D^{\alpha}u\left(x\right),...\right)
\end{eqnarray}
With the PDE (\ref{PDE}) one associates a mapping
\begin{eqnarray}
T:\mathcal{M}^{m}\left(\Omega\right)\rightarrow
\mathcal{M}^{0}\left(\Omega\right)\nonumber
\end{eqnarray}
where $\mathcal{M}^{m}\left(\Omega\right)$ is the space of
equivalence classes of functions which are continuously
differentiable up to order $m$ everywhere except on some closed
nowhere dense set \cite{Obergugenberger and Rosinger}, under the
equivalence relation
\begin{eqnarray}
u\sim v\Leftrightarrow \left(\begin{array}{ll}
\exists & \Gamma\subseteq\Omega\mbox{ closed nowhere dense :} \\
& \begin{array}{ll}
1) & u,v\in\mathcal{C}^{m}\left(\Omega\right) \\
2) & x\in\Omega\setminus\Gamma\Rightarrow v\left(x\right)=u\left(x\right) \\
\end{array} \\
\end{array}\right)
\end{eqnarray}
The mapping $T$ induces an equivalence relation $\sim_{T}$ on
$\mathcal{M}^{m}\left(\Omega\right)$ through
\begin{eqnarray}
\begin{array}{ll}
\forall & u,v\in\mathcal{M}^{m}\left(\Omega\right)\mbox{ :} \\
& u\sim_{T}v \Leftrightarrow Tu=Tv \\
\end{array}\label{TEqv}
\end{eqnarray}
With the mapping $T$ one associates in a canonical way an
injective mapping
\begin{eqnarray}
\widehat{T}:\mathcal{M}^{m}_{T}\left(\Omega\right) \rightarrow
\mathcal{M}^{0}\left(\Omega\right)\nonumber
\end{eqnarray}
where $\mathcal{M}^{m}_{T}\left(\Omega\right)$ denotes the
quotient space $\mathcal{M}^{m}\left(\Omega\right)/\sim_{T}$.  The
space $\mathcal{M}^{m}_{T}\left(\Omega\right)$ is ordered through
\begin{eqnarray}
\begin{array}{ll}
\forall & U,V\in\mathcal{M}^{m}_{T}\left(\Omega\right)\mbox{ :} \\
& U\leq_{T}V\Leftrightarrow\widehat{T}U\leq \widehat{T}V \\
\end{array}\label{TOrder}
\end{eqnarray}
so that $\widehat{T}$ is an order isomorphic embedding.  The
mapping $\widehat{T}$ extends uniquely to an order isomorphic
embedding
\begin{eqnarray}
\widetilde{T}^{\sharp}:\mathcal{M}^{m}_{T}\left(\Omega\right)^{\sharp}
\rightarrow \mathcal{M}^{0}\left(\Omega\right)^{\sharp}
\end{eqnarray}
where $\mathcal{M}^{m}_{T}\left(\Omega\right)^{\sharp}$ and
$\mathcal{M}^{0}\left(\Omega\right)^{\sharp}$ denote the Dedekind
order completions of $\mathcal{M}^{m}_{T}\left(\Omega\right)$ and
$\mathcal{M}^{0}\left(\Omega\right)^{\sharp}$, respectively.  This
is summarized in the following commutative diagram:\\
\begin{math}
\setlength{\unitlength}{1cm} \thicklines
%\begin{pspicture}(13,6)
\begin{picture}(13,6)
\put(2.4,5){$\mathcal{M}_{T}^{m}\left(\Omega\right)$}
\put(3.8,5.1){\vector(1,0){6.5}}
\put(10.5,5){$\mathcal{M}^{0}\left(\Omega\right)$}
\put(6.6,5.4){$\widehat{T}$} \put(3.0,4.8){\vector(0,-1){3.3}}
\put(2.4,1){$\mathcal{M}_{T}^{m}\left(\Omega\right)^{\sharp}$}
\put(4.0,1.1){\vector(1,0){6.2}}
\put(6.6,1.3){$\widehat{T}^{\sharp}$}
\put(10.5,1){$\mathcal{M}^{0}\left(\Omega\right)^{\sharp}$}
\put(11.2,4.8){\vector(0,-1){3.3}}
\end{picture}
\end{math}
\\
Subject to a mild assumption on the PDE (\ref{PDE}), one has
\begin{eqnarray}
\begin{array}{ll}
\forall & f\in\mathcal{C}^{0}\left(\Omega\right)\mbox{ :} \\
\exists ! & U^{\sharp}\in\mathcal{M}^{m}_{T}\left(\Omega\right)^{\sharp}\mbox{ :} \\
& \widehat{T}^{\sharp}U^{\sharp}=f \\
\end{array}\nonumber
\end{eqnarray}
where
$U^{\sharp}\in\mathcal{M}^{m}_{T}\left(\Omega\right)^{\sharp}$ is
the unique generalized solution to (\ref{PDE}).  The unique
generalized solution should be interpreted as the
\textit{totality} of all super solutions, sub solutions and exact
solutions to (\ref{PDE}).  Recently \cite{Anguelov and Rosinger 1}
it was shown that the generalized solutions to a PDE of the form
(\ref{PDE}) may be assimilated with usual Hausdorff continuous
functions, in the sense that there is an order isomorphism between
$\mathcal{M}^{m}_{T}\left(\Omega\right)^{\sharp}$ and the space
$\mathbb{H}_{nf}\left(\Omega\right)$ of all nearly finite
Hausdorff continuous functions.

Taking into account the universality of the existence and
regularity result just described, one may notice that there is a
large scope for further enrichment of the basic theory of Order
Completion \cite{Obergugenberger and Rosinger}.  In particular,
the following may serve as guidelines for such an enrichment.
\begin{eqnarray}
&\mbox{(A)}& \mbox{The space of generalized solutions to
(\ref{PDE}) may depend on the PDE}\nonumber\\
&&\mbox{operator $T\left(x,D\right)$}\nonumber \\
&\mbox{(B)}& \mbox{There is no differential structure on the space
of generalized solutions}\nonumber
\end{eqnarray}
In order to accommodate (A), one may do away with the equivalence
relation (\ref{TEqv}) on $\mathcal{M}^{m}\left(\Omega\right)$ and
consider a partial order other than (\ref{TOrder}), which does not
depend on the partial differential operator $T\left(X,D\right)$.
Indeed, somewhat in the spirit of Sobolev, one may consider the
partial order
\begin{eqnarray}
\begin{array}{ll}
\forall & u,v\in\mathcal{M}^{m}\left(\Omega\right)\mbox{ :} \\
& u\leq_{D}v \Leftrightarrow \left(\begin{array}{ll}
\forall & |\alpha|\leq m\mbox{ :} \\
& D^{\alpha}u\leq D^{\alpha}v \\
\end{array}\right)  \\
\end{array}
\end{eqnarray}
which could also solve (B).  However, such an approach presents
several difficulties.  In particular, the existence of generalized
solutions in the Dedekind completion of the partially ordered set
$\left(\mathcal{M}^{m}\left(\Omega\right),\leq_{D}\right)$ is not
clear.  In fact, the possibly nonlinear mapping $T$ associated
with the PDE (\ref{PDE}) cannot be extended to the Dedekind
completion in a unique and meaningful way, unless $T$ satisfies
some additional and rather restrictive conditions.  We mention
that the use of partial orders other than (\ref{TOrder}) was
investigated in \cite[Section 13]{Obergugenberger and Rosinger},
but the partial orders that are considered are still tied to the
PDE operator $T\left(XD\right)$.  Regarding (B), we may recall
that there is in general no connection between the usual order on
$\mathcal{M}^{m}\left(\Omega\right)$ and the derivatives of the
functions that are its elements.

\subsection{The Order Convergence Structure}

One possible way of going beyond the basic theory of Order
Completion is motivated by the fact that the process of taking the
supremum of a subset $A$ of a partially ordered set $X$ is
essentially a process of approximation.  Indeed,
\begin{eqnarray}
x_{0}=\sup A\nonumber
\end{eqnarray}
means that the set $A$ approximates $x_{0}$ arbitrarily close from
below.  Approximation, however, is essentially a topological
process.  Hence a topological type model for the process of
Dedekind completion of $\mathcal{M}^{0}\left(\Omega\right)$ may
serve as a starting point for the enrichment of the Order
Completion Method.

In this regard, we recall that there are several useful modes of
convergence on a partially ordered set, defined in terms of the
partial order, see for instance \cite{Birkhoff}, \cite{Luxemurg
and Zaanen} and \cite{Peressini}.  In particular, we consider the
order convergence of sequences defined on a partially ordered set
$X$ as
\begin{eqnarray}
&&\left(x_{n}\right)\mbox{ order converges to }x\in X
\Leftrightarrow\nonumber\\
&&\Leftrightarrow \left(\begin{array}{ll}
  \exists & \left(\lambda_{n}\right)\mbox{, }\left(\mu_{n}\right)\subset X\mbox{ :} \\
    & \begin{array}{ll}
      1) & n\in\mathbb{N}\Rightarrow \lambda_{n}\leq \lambda_{n+1}\leq x_{n}\leq \mu_{n+1}\leq \mu_{n} \\
      2) & \sup\left\{\lambda_{n}\mbox{ : }n\in\mathbb{N}\right\}=x=\inf\left\{\mu_{n}\mbox{ : }n\in\mathbb{N}\right\} \\
    \end{array} \\
\end{array}\right)\label{OCDef}
\end{eqnarray}
It is well known that the order convergence of sequences is in
general not topological, as is demonstrated in \cite{van der
Walt}.  That is, for a partially ordered set $X$ there is no
topology $\tau$ on $X$ such that the $\tau$-convergent sequences
are exactly the order convergent sequences.  However, see
\cite{Anguelov and van der Walt} and \cite{van der Walt 1}, for a
$\sigma$-distributive lattice $X$ there exists a convergence
structure $\lambda_{o}$, in the sense of \cite{Beattie and
Butzmann}, on $X$ that induces the order convergence of sequences
through
\begin{eqnarray}
\begin{array}{ll}
\forall & x\in X\mbox{ :} \\
\forall & \left(x_{n}\right)\subset X\mbox{ :} \\
& \left(x_{n}\right)\mbox{ order converges to }x\Leftrightarrow [\{\{x_{n}\mbox{ : }n\geq k\}\mbox{ : }k\in\mathbb{N}\}]\in\lambda_{o}\left(x\right) \\
\end{array}\label{SeqConv}
\end{eqnarray}
In particular, the order convergence structure, defined and
studied in \cite{Anguelov and van der Walt} and \cite{van der Walt
1} induces the order convergence of sequences through
(\ref{SeqConv}), and is defined as
\begin{eqnarray}
\begin{array}{ll}
  \forall & x\in X\mbox{ :} \\
  \forall & \mathcal{F}\mbox{ a filter on $X$ :} \\
    & \mathcal{F}\in\lambda_{o}\left(x\right) \Leftrightarrow \left(\begin{array}{ll}
      \exists & \left(\lambda_{n}\right)\mbox{, }\left(\mu_{n}\right)\subset X\mbox{ :} \\
        & \begin{array}{ll}
          1) & n\in\mathbb{N}\Rightarrow \lambda_{n}\leq \lambda_{n+1}\leq x \leq \mu_{n+1}\leq\mu_{n}\\
          2) & \sup\left\{\lambda_{n}\mbox{ : }n\in\mathbb{N}\right\} = x = \inf\left\{\mu_{n}\mbox{ : }n\in\mathbb{N}\right\} \\
          3) & [\left\{[\lambda_{n},\mu_{n}]\mbox{ : }n\in\mathbb{N}\right\}]\subseteq\mathcal{F} \\
        \end{array} \\
    \end{array}\right) \\
\end{array}\label{OCSDef}
\end{eqnarray}
and is Hausdorff, regular and first countable, see \cite{van der
Walt 1}.

A particular case of the above occurs when $X$ is an Archimedean
vector lattice.  In this case the convergence structure
$\lambda_{o}$ is a vector space convergence structure, and as such
it is induced by a uniform convergence structure, in the sense of
\cite{Gahler 1}. Indeed, the Cauchy filters are characterized as
\begin{eqnarray}
\begin{array}{ll}
\forall & \mathcal{F}\mbox{ a filter on $X$ :} \\
& \mathcal{F}\mbox{ is Cauchy }\Leftrightarrow \mathcal{F}-\mathcal{F}\in\lambda_{o}\left(x\right) \\
\end{array}\nonumber
\end{eqnarray}
The convergence vector space completion of an Archimedean vector
lattice $X$, equipped with the order convergence structure
$\lambda_{o}$ may be constructed as the Dedekind
$\sigma$-completion $X^{\sharp}$ of $X$, equipped with the order
convergence structure, see \cite{van der Walt 1}.  If $X$ is order
separable, then the completion of $X$ is in fact its Dedekind
completion.  In the particular case when
$X=\mathcal{C}\left(Y\right)$, with $Y$ a metric space, then the
convergence vector space completion is the set
$\mathbb{H}_{ft}\left(X\right)$ of finite Hausdorff continuous
functions on $Y$, which is the Dedekind completion of
$\mathcal{C}\left(Y\right)$.

Let us now consider the possibility of applying the above results
to the problem of solving nonlinear PDEs.  In this regard,
consider a nonlinear PDE of the form (\ref{PDE}), and the
associated mapping
\begin{eqnarray}
T:\mathcal{M}^{m}\left(\Omega\right)\rightarrow
\mathcal{M}^{0}\left(\Omega\right)\nonumber
\end{eqnarray}
The Order Completion Method is based on the abundance of
\textit{approximate solutions} to (\ref{PDE}), which are elements
of $\mathcal{M}^{m}\left(\Omega\right)$, and in general one cannot
expect these approximations to be continuous, let alone
\textit{sufficiently smooth}, on the whole of $\Omega$.  Moreover,
the space $\mathbb{H}_{ft}\left(\Omega\right)$ does not contain
the space $\mathcal{M}^{0}\left(\Omega\right)$.

On the other hand, the space $\mathcal{M}^{0}\left(\Omega\right)$
is an order separable Archimedean vector lattice
\cite{Obergugenberger and Rosinger}, and therefore one may equip
it with the order convergence structure.  The completion of this
space will be its Dedekind completion
$\mathcal{M}^{0}\left(\Omega\right)^{\sharp}$, as desired.
However, there are several obstacles.  If one equips
$\mathcal{M}^{m}\left(\Omega\right)$ with the subspace convergence
structure, then the nonlinear mapping $T$ is not necessarily
continuous.  Moreover, the quotient space
$\mathcal{M}^{m}_{T}\left(\Omega\right)$ is \textit{not} a linear
space, so that the completion process for convergence vector
spaces does not apply.  It is therefore necessary to develop a
\textit{nonlinear} topological model for the Dedekind completion
of $\mathcal{M}\left(\Omega\right)$.

\section{Spaces of Lower Semi-Continuous Functions}

The notion of a normal lower semi-continuous function,
respectively normal upper semi-continuous function, was introduced
by Dilworth \cite{Dilworth} in connection with the Dedekind
completion of spaces of continuous functions.  Dilworth introduced
the concept for \textit{bounded}, real valued functions.
Subsequently the definition was extended to \textit{locally
bounded} functions \cite{Anguelov et al}.  The definition extends
in a straight forward way to extended real valued functions.  In
particular, a function $u:X \rightarrow \overline{\mathbb{R}}$,
with $X$ a topological space, is normal lower semi-continuous
whenever
\begin{eqnarray}
\left(I\circ S\right)\left(u\right)\left(x\right) =
u\left(x\right)\mbox{, }x\in X\label{NLSDef}
\end{eqnarray}
where $I$ and $S$ are the Lower- and Upper Baire Operators, see
\cite{Anguelov}, \cite{Baire} and \cite{Sendov}, defined as
\begin{eqnarray}
I\left(u\right)\left(x\right)= \sup\{\inf\{u\left(y\right)\mbox{ :
}y\in V\}\mbox{ : }V\in\mathcal{V}_{x}\}\mbox{, }x\in
X\label{IDef}
\end{eqnarray}
\begin{eqnarray}
S\left(u\right)\left(x\right)= \inf\{\sup\{u\left(y\right)\mbox{ :
}y\in V\}\mbox{ : }V\in\mathcal{V}_{x}\}\mbox{, }x\in
X\label{SDef}
\end{eqnarray}
where $u$ is any extended real valued function on $X$.  A normal
lower semi-continuous function $u$ is called \textit{nearly
finite} if the set
\begin{eqnarray}
\{x\in X\mbox{ : }u\left(x\right)\in\mathbb{R}\}\nonumber
\end{eqnarray}
is open and dense in $X$.  We denote the space of all nearly
finite normal lower semi-continuous functions by
$\mathcal{NL}\left(X\right)$.  The space
$\mathcal{NL}\left(X\right)$ is ordered in a pointwise way through
\begin{eqnarray}
\begin{array}{ll}
\forall & u,v\in\mathcal{NL}\left(X\right)\mbox{ :} \\
& u\leq v\Leftrightarrow \left(\begin{array}{ll}
\forall & x\in X\mbox{ :} \\
& u\left(x\right)\leq v\left(x\right) \\
\end{array}\right) \\
\end{array}
\end{eqnarray}
The space $\mathcal{NL}\left(X\right)$ satisfies the following
properties.
\begin{theorem}\label{NLDedComp}
The space $\mathcal{NL}\left(X\right)$ is Dedekind complete.
Moreover, if $A\subseteq \mathcal{NL}\left(X\right)$ is bounded
from above, and $B\subseteq \mathcal{NL}\left(X\right)$ is bounded
from below, then
\begin{eqnarray}
\sup A= \left(I\circ S\right)\left(\phi\right) \nonumber
\end{eqnarray}
\begin{eqnarray}
\inf B= \left(I\circ S\circ I\right)\left(\varphi\right) \nonumber
\end{eqnarray}
where
\begin{eqnarray}
\phi:X\ni x\mapsto \sup\{u\left(x\right)\mbox{ : }u\in
A\}\nonumber
\end{eqnarray}
and
\begin{eqnarray}
\varphi:X\ni x\mapsto \inf\{u\left(x\right)\mbox{ : }u\in
B\}\nonumber
\end{eqnarray}
\end{theorem}
\begin{proof}
One may prove the result directly.  However, it is straight
forward to show that $\mathcal{NL}\left(X\right)$ is order
isomorphic to the set $\mathbb{H}_{nf}\left(\Omega\right)$ of
nearly finite Hausdorff continuous functions \cite{Anguelov et al
1}. The result follows immediately from the respective result in
\cite{Anguelov et al}.
\end{proof}\\ \\
Applying similar arguments, we obtain the following useful result.
\begin{proposition}\label{ConEqDense}
Consider any $u\in\mathcal{NL}\left(X\right)$.  Then there is a
set $U\subseteq X$ such that $X\setminus U$ is of First Baire
Category and $u\in\mathcal{C}\left(X\setminus U\right)$.  What is
more, if $v\in\mathcal{NL}\left(X\right)$ and $D\subseteq X$ is
dense in $X$, then
\begin{eqnarray}
\left(\begin{array}{ll}
\forall & x\in D\mbox{ :} \\
& u\left(x\right)\leq v\left(x\right) \\
\end{array}\right)\Rightarrow u\leq v \nonumber
\end{eqnarray}
\end{proposition}
\begin{proof}
Again a direct proof of this proposition is available. However the
result follows easily by considering the order isomorphism
\begin{eqnarray}
I:\mathbb{H}_{nf}\left(X\right)\rightarrow
\mathcal{NL}\left(X\right)\nonumber
\end{eqnarray}
\end{proof}\\ \\
\begin{proposition}\label{HnfFullyDistOrderSep}
The space $\mathcal{NL}\left(X\right)$ is fully distributive.
\end{proposition}
\begin{proof}
Consider a set $A\subset \mathcal{NL}\left(X\right)$ such that
\begin{eqnarray}
\sup A=u_{0}\nonumber
\end{eqnarray}
For $v\in\mathcal{NL}\left(X\right)$ we must show
\begin{eqnarray}
u_{0}\wedge v=\sup\{u\wedge v\mbox{ : }u\in A\}\label{Dist}
\end{eqnarray}
Suppose that (\ref{Dist}) fails for some $A\subset
\mathcal{NL}\left(X\right)$ and some $v\in
\mathcal{NL}\left(X\right)$.  That is,
\begin{eqnarray}
\begin{array}{ll}
\exists & w\in\mathcal{NL}\left(X\right)\mbox{ :} \\
& u\in A\Rightarrow u\wedge v\leq w< u_{0}\wedge v \\
\end{array}\label{NotDist}
\end{eqnarray}
Clearly, $u_{0},v\geq w$ so that there is some $u\in A$ such that
$w$ is not larger than $u$.  In view of Proposition
\ref{ConEqDense}
\begin{eqnarray}
\begin{array}{ll}
\exists & V\subseteq X\mbox{ nonempty, open :} \\
& x\in V\Rightarrow w\left(x\right)< u\left(x\right) \\
\end{array}\label{DistClaim}
\end{eqnarray}
Upon application of Proposition \ref{NLDedComp} we find
\begin{eqnarray}
\left(v\wedge u\right)\left(x\right)> u\left(x\right)\mbox{, }x\in
V\nonumber
\end{eqnarray}
since the operators $I$ and $S$ are monotone and idempotent
\cite[Section 2]{Anguelov}.  Hence (\ref{NotDist}) cannot hold.
This completes the proof.
\end{proof}\\ \\

The set $\mathcal{C}_{nd}\left(X\right)$ of all functions
$u:X\rightarrow \mathbb{R}$ that are continuous everywhere except
on some closed nowhere dense subset of $X$, that is,
\begin{eqnarray}
u\in\mathcal{C}_{nd}\left(X\right)\Leftrightarrow\left(\begin{array}{ll}
  \exists & \Gamma_{u}\subset X\mbox{ closed nowhere dense
:} \\
   & u\in\mathcal{C}\left(X\setminus\Gamma_{u}\right) \\
\end{array}\right)\nonumber
\end{eqnarray}
plays a fundamental role in the theory of Order Completion
\cite{Obergugenberger and Rosinger}, as discussed in the
introduction.  In particular, one considers the quotient space
$\mathcal{M}\left(X\right)=\mathcal{C}_{nd}\left(X\right)/\sim$,
where the equivalence relation $\sim$ on
$\mathcal{C}_{nd}\left(X\right)$ is defined by
\begin{eqnarray}
u\sim v\Leftrightarrow\left(\begin{array}{ll}
  \exists & \Gamma\subset X\mbox{ closed
nowhere dense :} \\
   & x\in X\setminus\Gamma\Rightarrow u\left(x\right)=v\left(x\right)
\end{array}\right)\label{CNDEq}
\end{eqnarray}
An order isomorphic representation of the space
$\mathcal{M}\left(X\right)$, consisting of normal lower
semi-continuous functions, is obtained by considering the set
\begin{eqnarray}
\mathcal{ML}\left(X\right)=\left\{u\in\mathcal{NL}\left(X\right)\mbox{
}\begin{array}{|ll}
\exists & \Gamma\subset X\mbox{ closed nowhere dense :} \\
& u\in\mathcal{C}\left(X\setminus\Gamma\right) \\
\end{array}\right\}
\end{eqnarray}
The advantage of considering the space
$\mathcal{ML}\left(X\right)$ in stead of
$\mathcal{M}\left(X\right)$ is that the elements of
$\mathcal{ML}\left(X\right)$ are actual point valued functions on
$X$, as apposed to the elements of $\mathcal{M}\left(X\right)$
which are equivalence classes of functions.  Hence the value
$u\left(x\right)$ of $u\in\mathcal{ML}\left(X\right)$ is
completely determined for every $x\in X$.
\begin{proposition}\label{OrderIsomorphism}
The mapping
\begin{eqnarray}
I_{S}:\mathcal{M}\left(X\right)\ni U\mapsto \left(I\circ
S\right)\left(u\right)\in \mathcal{ML}\left(X\right) \label{ISDef}
\end{eqnarray}
is a well defined order isomorphism.
\end{proposition}
\begin{proof}
First we show that the mapping $I_{S}$ is well defined.  In this
regard, consider some $U\in\mathcal{M}\left(X\right)$ and any
$u,v\in U$.  Let $\Gamma\subset X$ be the closed nowhere dense set
associated with $u$ and $v$ through (\ref{CNDEq}).  Since $\Gamma$
is closed, it follows by (\ref{IDef}) and (\ref{SDef}) that
\begin{eqnarray}
\left(I\circ S\right)\left(u\right)\left(x\right) = \left(I\circ
S\right)\left(v\right)\left(x\right)\mbox{, }x\in X\setminus
\Gamma\label{Equal}
\end{eqnarray}
Since $X\setminus \Gamma$ is dense in $X$, it follows that
\begin{eqnarray}
\begin{array}{ll}
\forall & x\in X\mbox{ :} \\
\forall & V_{1},V_{2}\in\mathcal{V}_{x}\mbox{ :} \\
\exists & x_{0}\in X\mbox{ :} \\
& x_{0}\in \left(X\setminus\Gamma\right)\cap\left(V_{1}\cap V_{2}\right) \\
\end{array}\nonumber
\end{eqnarray}
For any $x\in X$ we have
\begin{eqnarray}
\inf\left\{\left(I\circ S\right)\left(u\right)\left(y\right)\mbox{
: }y\in V_{1}\right\}\leq \left(I\circ
S\right)\left(u\right)\left(x_{0}\right)\nonumber
\end{eqnarray}
and
\begin{eqnarray}
\left(I\circ S\right)\left(v\right)\left(x_{0}\right) \leq
\sup\left\{\left(I\circ S\right)\left(v\right)\left(y\right)\mbox{
: }y\in V_{2}\right\}\nonumber
\end{eqnarray}
Hence it follows by (\ref{Equal}) that
\begin{eqnarray}
\inf\left\{\left(I\circ S\right)\left(u\right)\left(y\right)\mbox{
: }y\in V_{1}\right\}\leq \sup\left\{\left(I\circ
S\right)\left(v\right)\left(y\right)\mbox{ : }y\in
V_{2}\right\}\nonumber
\end{eqnarray}
so that (\ref{IDef}) and (\ref{SDef}) yields
\begin{eqnarray}
I\left(\left(I\circ S\right)\left(u\right)\right)\leq
S\left(\left(I\circ S\right)\left(v\right)\right)\label{Ineq}
\end{eqnarray}
It now follows form the idempotency and monotonicity of the
operator $I$ \cite[Section 2]{Anguelov} that
\begin{eqnarray}
\left(I\circ S\right)\left(u\right)\leq \left(I\circ
S\right)\left( \left(I\circ S\right)\left(v\right)\right)\nonumber
\end{eqnarray}
Since the operator $\left(I\circ S\right)$ is also idempotent, see
\cite[Section ]{Anguelov and Rosinger 1}, one obtains
\begin{eqnarray}
\left(I\circ S\right)\left(u\right)\leq \left(I\circ
S\right)\left(v\right)\nonumber
\end{eqnarray}
By similar arguments it follows that
\begin{eqnarray}
\left(I\circ S\right)\left(v\right)\leq \left(I\circ
S\right)\left(u\right)\nonumber
\end{eqnarray}
so that $\left(I\circ S\right)\left(u\right)=\left(I\circ
S\right)\left(v\right)$.\\
It is obvious that the mapping $I_{S}$ is surjective.  To see that
it is injective, consider any $U,V\in\mathcal{M}\left(X\right)$.
Then we may assume that
\begin{eqnarray}
\begin{array}{ll}
\exists & A\subseteq X\mbox{ nonempty, open :} \\
\exists & \epsilon>0\mbox{ :} \\
\forall & u\in U\mbox{, }v\in V\mbox{ :} \\
& \begin{array}{ll}
1) & x\in A\Rightarrow u\left(x\right)< v\left(x\right)-\epsilon \\
2) & u,v\in\mathcal{C}\left(A\right) \\
\end{array} \\
\end{array}
\end{eqnarray}
so that
\begin{eqnarray}
I_{S}\left(U\right)\left(x\right)<
I_{S}\left(V\right)\left(x\right)-\epsilon\mbox{, }x\in A\nonumber
\end{eqnarray}
It remains to verify
\begin{eqnarray}
\begin{array}{ll}
  \forall & U,V\in\mathcal{M}\left( X\right)\mbox{ :} \\
    & U\leq V\Leftrightarrow I_{S}\left(U\right)\leq
I_{S}\left(V\right) \\
\end{array}\nonumber
\end{eqnarray}
The implication `$U\leq V\Rightarrow I_{S}\left(U\right)\leq
I_{S}\left(V\right)$' follows by similar arguments as those
employed to show that $I_{S}$ is well defined.  Conversely,
suppose that $I_{S}U\leq I_{S}V$ for some
$U,V\in\mathcal{M}\left(X\right)$. The result now follows in the
same way as the injectivity of $I_{S}$. This completes the proof.
\end{proof}\\
The following is now immediate.
\begin{corollary}\label{DistLat}
The space $\mathcal{ML}\left(X\right)$ is a fully distributive
lattice.
\end{corollary}

\section{The Uniform Order Convergence Structure on $\mathcal{ML}\left(X\right)$}

As a consequence of Proposition \ref{HnfFullyDistOrderSep} one may
define the order convergence structure $\lambda_{o}$ on the space
$\mathcal{ML}\left(X\right)$. The order convergence structure
induces the order convergence of sequences on
$\mathcal{ML}\left(X\right)$ and is Hausdorff, regular and first
countable.  In order to define a uniform convergence structure, in
the sense of \cite{Beattie and Butzmann}, we introduce the
following notation.  For any open subset $U$ of $X$, and any
subset $F$ of $\mathcal{ML}\left(X\right)$, we denote by $F_{|U}$
the restriction of $F$ to $U$.  That is,
\begin{eqnarray}
F_{|U}=\left\{v\in\mathcal{ML}\left(U\right)\mbox{ $|$
}\begin{array}{ll}
\exists & w\in F\mbox{ :} \\
& x\in U\Rightarrow w\left(x\right)=v\left(x\right) \\
\end{array}\right\}\nonumber
\end{eqnarray}
\begin{definition}\label{UOCSDef}
Let $\tau$ be the topology on $X$, and let $\Sigma$ consist of all
nonempty order intervals in $\mathcal{ML}\left(X\right)$.  Let
$\mathcal{J}_{o}$ denote the family of filters on
$\mathcal{ML}\left(X\right)\times \mathcal{ML}\left(X\right)$ that
satisfy the following:  There exists $k\in\mathbb{N}$ such that
\begin{eqnarray}
\begin{array}{ll}
    \forall & i=1,...,k\mbox{ :} \\
    \exists & \Sigma_{i}=\left(I_{n}^{i}\right)\subseteq\Sigma\mbox{ :} \\
      & \begin{array}{ll}
        1) & I_{n+1}^{i}\subseteq I_{n}^{i}\mbox{, }n\in\mathbb{N} \\
         2) & \left([\Sigma_{1}]\times [\Sigma_{1}]\right)\cap...
         \cap
         \left([\Sigma_{k}]\times [\Sigma_{k}]\right)
         \subseteq\mathcal{U}
      \end{array} \\
  \end{array}\label{DefCon1}
\end{eqnarray}
where $[\Sigma_{i}]=[\left\{F\mbox{ : }F\in\Sigma_{i}\right\}]$.
Moreover, for every $i=1,...,k$  and $V\in\tau$ one has
\begin{eqnarray}
\begin{array}{lll}
\begin{array}{ll}
  \exists & u_{i}\in\mathcal{ML}\left(X\right)\mbox{ :} \\
          & \cap_{n\in\mathbb{N}}I_{n}^{i|V}=\left\{u_{i}\right\}_{|V}
\end{array} & \mbox{ or } & \cap_{n\in\mathbb{N}}I_{n}^{i|V}=\emptyset \\
\end{array}\label{DefCon2}
\end{eqnarray}
\end{definition}
\begin{theorem}\label{UOCStructure}
The family $\mathcal{J}_{o}$ of filters on
$\mathcal{ML}\left(X\right)\times \mathcal{ML}\left(X\right)$
constitutes a uniform convergence structure.
\end{theorem}
\begin{proof}
The firs four axioms \cite[Definition 2.1.2]{Beattie and Butzmann}
are clearly fulfilled, so it remains to verify
\begin{eqnarray}
\begin{array}{ll}
\forall & \mathcal{U},\mathcal{V}\in\mathcal{J}_{o}\mbox{ :} \\
& \mathcal{U}\circ \mathcal{V}\mbox{ exists }\Rightarrow \mathcal{U}\circ \mathcal{V}\in\mathcal{J}_{o} \\
\end{array}\label{ClaimComp}
\end{eqnarray}
So take any $\mathcal{U},\mathcal{V}\in\mathcal{J}_{o}$ such that
$\mathcal{U}\circ \mathcal{V}$ exists, and let
$\Sigma_{1},...,\Sigma_{k}$ and $\Sigma_{1}',...,\Sigma_{l}'$ be
the collection of order intervals associated with $\mathcal{U}$
and $\mathcal{V}$, respectively, through Definition \ref{UOCSDef}.
Set
\begin{eqnarray}
\Phi=\{\left(i,j\right)\mbox{ : }[\Sigma_{i}]\circ
[\Sigma_{j}']\mbox{ exists}\}\nonumber
\end{eqnarray}
Then
\begin{eqnarray}
\mathcal{U}\circ\mathcal{V}\supseteq
\bigcap\{\left([\Sigma_{i}]\times [\Sigma_{i}]\right)\circ
\left([\Sigma_{j}]\times [\Sigma_{j}]\right)\mbox{ :
}\left(i,j\right)\in\Phi\}
\end{eqnarray}\label{Comp}
by \cite[Proposition 2.1.1 (i)]{Beattie and Butzmann}.  Now,
$\left(i,j\right)\in\Phi$ exists if and only if
\begin{eqnarray}
\begin{array}{ll}
\forall & m,n\in\mathbb{N}\mbox{ :} \\
& I_{m}^{i}\cap I_{n}^{j}\neq \emptyset \\
\end{array}\nonumber
\end{eqnarray}
For any $\left(i,j\right)\in \Phi$, set
$\Sigma_{i,j}=\left(I_{n}^{i,j}\right)$ where, for each
$n\in\mathbb{N}$
\begin{eqnarray}
I_{n}^{i,j}=[\inf\left(I_{n}^{i}\right)\wedge
\inf\left(I_{n}^{j}\right), \sup\left(I_{n}^{i}\right)\vee
\sup\left(I_{n}^{j}\right)]\nonumber
\end{eqnarray}
Now, using (\ref{Comp}), we find
\begin{eqnarray}
\mathcal{U}\circ \mathcal{V}\supseteq \bigcap\{[\Sigma_{i}]\times
[\Sigma_{j}]\mbox{ : }\left(i,j\right)\in\Phi\} \supseteq \bigcap
\{[\Sigma_{i,j}]\times [\Sigma_{i,j}]\mbox{ :
}\left(i,j\right)\in\Phi\}\nonumber
\end{eqnarray}
Clearly each $\Sigma_{i,j}$ satisfies 1) of (\ref{DefCon1}). Since
$\mathcal{ML}\left(X\right)$ is fully distributive, see Corollary
\ref{DistLat}, (\ref{DefCon2}) also holds.  This completes the
proof.
\end{proof}\\ \\
An important fact to note is that the uniform order convergence
structure $\mathcal{J}_{o}$ is defined solely in terms of the
order on $\mathcal{ML}\left(X\right)$, and the topology on $X$.
This is unusual for a uniform convergence structure on a function
space. Indeed, for a space of functions $F\left(X,Y\right)$,
defined on some set $X$, and taking values in $Y$, one defines the
uniform convergence structure either in terms of the uniform
convergence structure on $Y$, or in terms of a convergence
structure on $F\left(X,Y\right)$ which is suitably compatible with
the algebraic structure of the space. Indeed, a convergence vector
space carries a natural uniform convergence structure, where the
Cauchy filters are determined by the linear structure. That is,
\begin{equation}
\mathcal{F}\mbox{ a Cauchy filter}\Leftrightarrow
\mathcal{F}-\mathcal{F}\rightarrow 0
\end{equation}
This is also the case for the order convergence structure studied
in \cite{Anguelov and van der Walt} and \cite{van der Walt 1}. The
motivation for introducing a uniform convergence structure that
does not depend on the algebraic structure of the set
$\mathcal{ML}\left(X\right)$ comes from nonlinear PDEs, and in
particular the Order Completion Method \cite{Obergugenberger and
Rosinger}, as explained in the Introduction.

The convergence structure $\lambda_{\mathcal{J}_{o}}$ induced on
$\mathcal{ML}\left(X\right)$ by the uniform convergence structure
$\mathcal{J}_{o}$ may be characterized as follows.
\begin{theorem}\label{ConvChar}
A filter $\mathcal{F}$ on $\mathcal{ML}\left(X\right)$ belongs to
$\lambda_{\mathcal{J}_{o}}\left(u\right)$, for some
$u\in\mathcal{ML}\left(X\right)$, if and only if there exists a
family $\Sigma_{\mathcal{F}}=\left(I_{n}\right)$ of nonempty order
intervals on $\mathcal{ML}\left(X\right)$ such that
\begin{eqnarray}
\begin{array}{ll}
1) & I_{n+1}\subseteq I_{n}\mbox{, }n\in\mathbb{N} \\
2) & \begin{array}{ll}
\forall & V\in\tau\mbox{ :} \\
& \cap_{n\in\mathbb{N}}I_{n|V}=\{u\}_{|V} \\
\end{array} \\
\end{array}
\end{eqnarray}
and $[\Sigma_{\mathcal{F}}]\subseteq\mathcal{F}$.
\end{theorem}
\begin{proof}
Let the filter $\mathcal{F}$ converge to
$u\in\mathcal{ML}\left(X\right)$.  Then, by \cite[Definition
2.1.3]{Beattie and Butzmann},
$[u]\times\mathcal{F}\in\mathcal{J}_{o}$. Hence by Definition
\ref{UOCSDef} there exist $k\in\mathbb{N}$ and
$\Sigma_{i}\subseteq\Sigma$ for $i=1,...,k$ such that
(\ref{DefCon1}) through (\ref{DefCon2}) are satisfied.\\
Set $\Psi=\left\{i:[\Sigma_{i}]\subset[u]\right\}$. We claim
\begin{equation}\label{Claim}
\mathcal{F}\supset\bigcap_{i\in\Psi}\mathcal{I}^{i}
\end{equation}
Take a set $A\in\cap_{i\in\Psi}\mathcal{I}^{i}$.  Then for each
$i\in\Psi$ there is a set $A_{i}\in\mathcal{I}^{i}$ such that
$A\supset\cup_{i\in\Psi}A_{i}$.  For each
$i\in\left\{1,...,k\right\}\setminus\Psi$ choose a set
$A_{i}\in\mathcal{I}^{i}$ with $u\in
\mathcal{ML}\left(X\right)\setminus A_{i}$. Then
\begin{eqnarray}
\left(A_{1}\times A_{1}\right)\cup...\cup\left(A_{k}\times
A_{k}\right)\in\left(\mathcal{I}_{1}\times\mathcal{I}_{1}\right)\cap...\cap
\left(\mathcal{I}_{k}\times\mathcal{I}_{k}\right)\subset\mathcal{F}\times[u]\nonumber
\end{eqnarray}
and so there is a set $B\in\mathcal{F}$ such that
\begin{eqnarray}
B\times\left\{u\right\}\subset\left(A_{1}\times
A_{1}\right)\cup...\cup\left(A_{k}\times A_{k}\right)\nonumber
\end{eqnarray}
If $w\in B$ then $\left(u,w\right)\in A_{i}\times A_{i}$ for some
$i$.  Since $u\in A_{i}$, we get $i\in\Psi$ and so
$w\in\cup_{i\in\Psi}A_{i}$.  This gives
$B\subseteq\cup_{i\in\Psi}A_{i}\subseteq A$ and so
$A\in\mathcal{F}$ so that (\ref{Claim}) holds.\\
Clearly, for each $i\in\Psi$, we have
\begin{eqnarray}
\begin{array}{ll}
\forall & V\in\tau\mbox{ :} \\
& \cap_{n\in\mathbb{N}}I^{i}_{n|V}=\{u\}_{|V} \\
\end{array}\label{AAAA}
\end{eqnarray}
Writing each $I_{n}^{i}\in\Sigma_{i}$ in the form
$I_{n}^{i}=[\lambda_{n}^{i},\mu_{n}^{i}]$, we claim
\begin{eqnarray}
\sup\{\lambda_{n}^{i}\mbox{ : }n\in \mathbb{N}\}=u =
\inf\{\mu_{n}^{i}\mbox{ : }n\in \mathbb{N}\}
\end{eqnarray}
Suppose this were not the case.  Then there exists
$v,w\in\mathcal{ML}\left(X\right)$ such that
\begin{eqnarray}
\lambda_{n}\leq v < w \leq \mu_{n}\mbox{, }n\in\mathbb{N}\nonumber
\end{eqnarray}
Then, in view of Proposition \ref{ConEqDense}, there is some
nonempty $V\in\tau$ such that
\begin{eqnarray}
v\left(x\right)<w\left(x\right)\mbox{, }x\in V\nonumber
\end{eqnarray}
which contradicts (\ref{DefCon2}).  Since
$\mathcal{ML}\left(X\right)$ is fully distributive, the result
follows upon setting
\begin{eqnarray}
\Sigma_{\mathcal{F}}= \left\{[\lambda_{n},\mu_{n}]\mbox{ :
}\begin{array}{ll}
1) & \lambda_{n}=\inf\{\lambda_{n}^{i}\mbox{ : }i\in\Psi\} \\
2) & \mu_{n}=\sup\{\mu_{n}^{i}\mbox{ : }i\in\Psi\} \\
\end{array}\right\}
\end{eqnarray}
The converse is trivial.
\end{proof}\\ \\
The following is now immediate
\begin{corollary}
Consider a filter $\mathcal{F}$ on $\mathcal{ML}\left(X\right)$.
Then $\mathcal{F}\in\lambda_{\mathcal{J}_{o}}\left(u\right)$ if
and only if $\mathcal{F}\in\lambda_{o}\left(u\right)$.  Therefore
$\mathcal{ML}\left(X\right)$ is a uniformly Hausdorff uniform
convergence space.\\
In particular, a sequence $\left(u_{n}\right)$
on $\mathcal{ML}\left(X\right)$ converges to $u$ if and only if
$\left(u_{n}\right)$ order converges to $u$.
\end{corollary}

\section{The Completion of $\mathcal{ML}\left(X\right)$}

This section is concerned with constructing the completion of the
uniform convergence space $\mathcal{ML}\left(X\right)$.  In this
regard, recall that the completion of the convergence vector space
$\mathcal{C}\left(X\right)$, equipped with the order convergence
structure, is the set of finite Hausdorff continuous functions on
$X$ \cite{Anguelov and van der Walt}.  This space is order
isomorphic to the set a all \textit{finite} normal lower
semi-continuous functions.  Note, however, that functions
$u\in\mathcal{ML}\left(X\right)$ need not be finite everywhere,
but may, in contradistinction to functions in
$\mathcal{C}\left(X\right)$, assume the values $\pm \infty$ on any
closed nowhere dense subset of $X$.  Hence we consider the space
$\mathcal{NL}\left(X\right)$ of nearly finite normal lower
semi-continuous functions on $X$.  Following the results in
Section 3, we introduce the following uniform convergence
structure on $\mathcal{NL}\left(X\right)$.
\begin{definition}\label{UOCSNL}
Let $\tau$ be the topology on $X$, and let $\Sigma$ consist of all
nonempty order intervals in $\mathcal{NL}\left(X\right)$.  Let
$\mathcal{J}_{o}^{\sharp}$ denote the family of filters on
$\mathcal{NL}\left(X\right)\times \mathcal{NL}\left(X\right)$ that
satisfy the following:  There exists $k\in\mathbb{N}$ such that
\begin{eqnarray}
\begin{array}{ll}
    \forall & i=1,...,k\mbox{ :} \\
    \exists & \Sigma_{i}=\left(I_{n}^{i}\right)\subseteq\Sigma\mbox{ :} \\
      & \begin{array}{ll}
        1) & I_{n+1}^{i}\subseteq I_{n}^{i}\mbox{, }n\in\mathbb{N} \\
         2) & \left([\Sigma_{1}]\times [\Sigma_{1}]\right)\cap...
         \cap
         \left([\Sigma_{k}]\times [\Sigma_{k}]\right)
         \subseteq\mathcal{U}
      \end{array} \\
  \end{array}\label{DefCon1A}
\end{eqnarray}
where $[\Sigma_{i}]=[\left\{F\mbox{ : }F\in\Sigma_{i}\right\}]$.
Moreover, for every $i=1,...,k$  and $V\in\tau$ one has
\begin{eqnarray}
\begin{array}{lll}
\begin{array}{ll}
  \exists & u_{i}\in\mathcal{NL}\left(X\right)\mbox{ :} \\
          & \cap_{n\in\mathbb{N}}I_{n}^{i|V}=\left\{u_{i}\right\}_{|V}
\end{array} & \mbox{ or } & \cap_{n\in\mathbb{N}}I_{n}^{i|V}=\emptyset \\
\end{array}\label{DefCon2A}
\end{eqnarray}
\end{definition}
The following now follows by similar arguments as those employed
in Section 3.
\begin{theorem}
The family $\mathcal{J}_{o}^{\sharp}$ of filters on
$\mathcal{NL}\left(X\right)\times \mathcal{NL}\left(X\right)$ is a
Hausdorff uniform convergence structure.
\end{theorem}
\begin{theorem}\label{NLConv}
A filter $\mathcal{F}$ on $\mathcal{ML}\left(X\right)$ belongs to
$\lambda_{\mathcal{J}_{o}}$ if and only if
$\mathcal{F}\in\lambda_{o}\left(u\right)$.
\end{theorem}

We now proceed to show that $\mathcal{NL}\left(X\right)$ is the
completion of $\mathcal{ML}\left(X\right)$.  That is, we show that
the following three conditions are satisfied:
\begin{itemize}
    \item The uniform convergence space
    $\mathcal{NL}\left(X\right)$ is complete
    \item $\mathcal{ML}\left(X\right)$ is uniformly isomorphic to
    a dense subspace of $\mathcal{NL}\left(X\right)$
    \item Any uniformly continuous mapping $\varphi$ on
    $\mathcal{ML}\left(X\right)$ into a complete, Hausdorff uniform convergence space
    $Y$ extends uniquely to a uniformly continuous mapping $\varphi^{\sharp}$
    from $\mathcal{NL}\left(X\right)$ into $Y$.
\end{itemize}
\begin{proposition}
The uniform convergence space $\mathcal{NL}\left(X\right)$ is
complete.
\end{proposition}
\begin{proof}
Let $\mathcal{F}$ be a Cauchy filter on
$\mathcal{NL}\left(X\right)$, so that $\mathcal{F}\times
\mathcal{F}\in\mathcal{J}_{o}^{\sharp}$. Let
$\Sigma_{1},...,\Sigma_{k}$ be the families of order intervals
associated with $\mathcal{F}\times \mathcal{F}$ through Definition
\ref{UOCSNL}.  Since $\mathcal{NL}\left(X\right)$ is Dedekind
complete it follows by (\ref{DefCon2A}) that, for each $i=1,...,k$
\begin{eqnarray}
\sup\{\lambda_{n}^{i}\mbox{ : }n\in\mathbb{N}\}=u_{i}
=\inf\{\mu_{n}^{i}\mbox{ : }n\in\mathbb{N}\}
\end{eqnarray}
for some $u_{i}\in\mathcal{NL}\left(X\right)$, where
$I_{n}^{i}=[\lambda^{i}_{n},\mu^{i}_{n}]$ for each
$n\in\mathbb{N}$. By Theorem \ref{NLConv} each of the filters
$\mathcal{F}_{i}=[\Sigma_{i}]$ converges to $u_{i}$. Let
$\mathcal{G}\supseteq \mathcal{F}$ be an ultrafilter.  Since
\begin{eqnarray}
\mathcal{F}\supseteq \mathcal{F}_{1}\cap...\cap\mathcal{F}_{k}
\nonumber
\end{eqnarray}
it follows that $\mathcal{G}\supseteq\mathcal{F}_{i}$ for at least
one $i=1,...,k$, so that $\mathcal{G}$ converges to $u_{i}$.
Therefore \cite[Proposition 2.3.2 (iii)]{Beattie and Butzmann} the
filter $\mathcal{F}$ converges to $u_{i}$.  This completes the
proof.
\end{proof}\\ \\
\begin{theorem}\label{Completion}
Let $X$ be a metric space.  Then the space
$\mathcal{NL}\left(X\right)$ is the uniform convergence space
completion of $\mathcal{ML}\left(X\right)$.
\end{theorem}
\begin{proof}
First we show that the identity mapping
$\iota:\mathcal{ML}\left(X\right)\rightarrow
\mathcal{NL}\left(X\right)$ is a uniformly continuous embedding.
In this regard, it is sufficient to consider a filter
$[\Sigma_{\mathcal{F}}]$ where $\Sigma_{\mathcal{F}}$ is a family
of nonempty order intervals in $\mathcal{ML}\left(X\right)$ that
satisfies 1) of (\ref{DefCon1}) and (\ref{DefCon2}).  Clearly
\begin{eqnarray}
\begin{array}{ll}
\forall & I_{n}=[\lambda_{n},\mu_{n}]\in\Sigma_{\mathcal{F}}\mbox{ :} \\
& \iota\left(I_{n}\right)\subseteq [\iota\left(\lambda_{n}\right),\iota\left(\mu_{n}\right)] \\
\end{array}
\end{eqnarray}
The family
\begin{eqnarray}
\Sigma_{\iota\left(\mathcal{F}\right)}=
\left(I_{n}'\right)=\left\{[\iota\left(\lambda_{n}\right),\iota\left(\mu_{n}\right)]\mbox{
: }n\in\mathbb{N}\right\}
\end{eqnarray}
satisfies 1) of (\ref{DefCon1A}).  To see that (\ref{DefCon2A})
holds, we proceed by contradiction.  Assume that for some $W\in
\tau$
\begin{eqnarray}
\begin{array}{ll}
\exists & u,v\in\mathcal{NL}\left(X\right)\mbox{ :} \\
& \cap_{n\in\mathbb{N}}I_{n|W}'\supseteq \{u,v\}_{|W} \\
\end{array}\label{BBB}
\end{eqnarray}
where $u_{|W}\neq v_{|W}$.  We may assume that
$u\left(x\right)<v\left(x\right)$, $x\in W$. Clearly,
\begin{eqnarray}
\lambda_{n}\left(x\right)\leq \varphi\left(x\right)\leq
u\left(x\right)<v\left(x\right)\leq \mu_{n}\mbox{, }x\in W
\end{eqnarray}
for every $n\in\mathbb{N}$, where
\begin{eqnarray}
\varphi\left(x\right)=\sup\{\lambda_{n}\left(x\right)\mbox{ :
}n\in\mathbb{N}\}\nonumber
\end{eqnarray}
which is upper semi-continuous. Applying Hahn's Theorem twice we
find
\begin{eqnarray}
\begin{array}{ll}
\exists & \phi,\psi\in\mathcal{C}\left(W\right)\mbox{ :} \\
& \{\phi,\psi\}\subseteq \cap_{n\in\mathbb{N}}I_{n|W} \\
\end{array}\nonumber
\end{eqnarray}
which contradicts (\ref{DefCon2}) so that (\ref{DefCon2A}) must
hold. That $\iota^{-1}$ is uniformly continuous is trivial.\\
To see that $\iota\left(\mathcal{ML}\left(X\right)\right)$ is
dense in $\mathcal{NL}\left(X\right)$, consider any
$u\in\mathcal{NL}\left(X\right)$, and set
\begin{eqnarray}
D_{u}=\{x\in X\mbox{ : }u\left(x\right)\in\mathbb{R}\}\nonumber
\end{eqnarray}
Since $D_{u}$ is open, it follows that $u$ restricted to $D_{u}$
is normal lower semi-continuous.  Since $u$ is also finite on
$D_{u}$ it follows, see \cite[Proof of Theorem 26]{Anguelov and
van der Walt} that there exists a sequence $\left(u_{n}\right)$ of
continuous functions on $D_{u}$ such that
\begin{eqnarray}
u\left(x\right)=\sup\{u_{n}\left(x\right)\mbox{ :
}n\in\mathbb{N}\}\mbox{, }x\in D_{u}\label{sup}
\end{eqnarray}
Consider now the sequence $\left(v_{n}\right)=\left(\left(I\circ
S\right)\left(u_{n}^{0}\right)\right)$ where
\begin{eqnarray}
u_{n}^{0}\left(x\right)=\left\{\begin{array}{lll}
u_{n}\left(x\right) & \mbox{ if }& x\in D_{u} \\
0 & \mbox{ if } & x\notin D_{u} \\
\end{array}\right.
\end{eqnarray}
Clearly $v_{n}\left(x\right)=u_{n}\left(x\right)$ for every $x\in
D_{u}$.  We claim
\begin{eqnarray}
u=sup \{v_{n}\mbox{ : }n\in\mathbb{N}\}\label{DClaim}
\end{eqnarray}
If (\ref{DClaim}) does not hold, then
\begin{eqnarray}
\begin{array}{ll}
\exists & v\in\mathcal{NL}\left(X\right)\mbox{ :} \\
& n\in\mathbb{N}\Rightarrow v_{n}\leq v<u \\
\end{array}\nonumber
\end{eqnarray}
But then, in view of Proposition \ref{ConEqDense}, and the fact
that $D_{u}$ is open and dense, there exists an open and nonempty
set $W\subseteq D_{u}$ such that
\begin{eqnarray}
\begin{array}{ll}
\forall & x\in W\mbox{ :} \\
& n\in\mathbb{N}\Rightarrow u_{n}\left(x\right)\leq
v\left(x\right)< u\left(x\right)
\end{array}\nonumber
\end{eqnarray}
which contradicts (\ref{sup}).  Therefore (\ref{DClaim}) must
hold.  The sequence $\left(v_{n}\right)$ is clearly a Cauchy
sequence in $\mathcal{ML}\left(X\right)$ so that
$\mathcal{ML}\left(X\right)$ is dense in
$\mathcal{NL}\left(X\right)$.\\
The extension property for uniformly continuous mappings on
$\mathcal{ML}\left(X\right)$ follows in the standard way.
\end{proof}\\ \\
Note that in the above proof, we actually showed that
$\mathcal{NL}\left(X\right)$ is the Dedekind completion of
$\mathcal{ML}\left(X\right)$.  Hence the uniform order convergence
structure provides a nonlinear topological model for the process
of taking the Dedekind completion of $\mathcal{ML}\left(X\right)$.
In view of Proposition \ref{OrderIsomorphism}, this extends a
previous result of Anguelov \cite{Anguelov} on the Dedekind
completion of $\mathcal{M}\left(X\right)$.

\section{An Application to Nonlinear PDEs}

As an illustration of how the results developed in this paper may
be applied to the problem of obtaining generalized solutions to
nonlinear PDEs, we consider the Navier-Stokes equations in three
spatial dimensions given by
\begin{eqnarray}
\begin{array}{l}
\frac{\partial}{\partial t}u_{i}\left(x,t\right) +
\sum_{j=1}^{3}u_{j}\left(x,t\right)\frac{\partial}{\partial
x_{i}}u_{j}\left(x,t\right) -\nu
\sum_{j=1}^{3}\frac{\partial^{2}}{\partial x_{j}^{2}}
u_{i}\left(x,t\right) +\frac{\partial p}{\partial
x_{i}}\left(x,t\right)  =
f\left(x,t\right) \\
\\
\sum_{i=1}^{3}\frac{\partial}{\partial x_{i}}u_{i}\left(x,t\right) = 0 \\
\end{array}\label{NSEq}
\end{eqnarray}
where $\left(x,t\right)\in\Omega=\mathbb{R}^{3}\times [0,\infty)$,
and $f\in\mathcal{C}^{0}\left(\Omega,\mathbb{R}^{3}\right)$. We
also require the unknown function
$u=\left(u_{1},u_{2},u_{3}\right)$ to satisfy the initial value
\begin{eqnarray}
u\left(x,0\right)=u^{0}\left(x\right)\mbox{,
}x\in\mathbb{R}^{3}\label{Initval}
\end{eqnarray}
where $u^{0}\in\mathcal{C}^{2}\left(\mathbb{R}^{3},
\mathbb{R}^{3}\right)$ is a given, divergence free vector field.
The equations (\ref{NSEq}) are supposed to model the motion of a
fluid through three dimensional space, where $u$ specifies the
velocity, and $p$ the pressure in the fluid.   We write the
equation (\ref{NSEq}) in the compact form
\begin{eqnarray}
T\left(x,t,D\right)v\left(x,t\right)=g\left(x,t\right)\mbox{, }
\left(x,t\right)\in\Omega\nonumber
\end{eqnarray}
where $v=\left(u,p\right)$, $g=\left(f,0\right)$ and the nonlinear
PDE operator $T\left(x,t,D\right)$ is defined through a continuous
mapping $F:\Omega\times\mathbb{R}^{K} \rightarrow \mathbb{R}^{4}$
by
\begin{eqnarray}
T\left(x,t,D\right)v\left(x,t\right)=
F\left(x,t,v\left(x,t\right),...,D^{\alpha}
v\left(x,t\right),...\right)\mbox{, }|\alpha|\leq 2\label{PDEOp}
\end{eqnarray}
With the system of PDEs (\ref{NSEq}) we can associate a mapping
\begin{eqnarray}
T:\mathcal{C}^{2}\left(\Omega\right)^{4}\ni u \mapsto
\left(T_{1}u,T_{2}u,T_{3}u,T_{4}u\right)\in
\mathcal{C}^{0}\left(\Omega\right)^{4} \label{AssMap}
\end{eqnarray}
In view of (\ref{PDEOp}), one may extend the mappings $T$ uniquely
to
\begin{eqnarray}
T:\mathcal{C}^{2}_{nd}\left(\Omega\right)^{4}\rightarrow
\mathcal{C}^{0}_{nd}\left(\Omega\right)\nonumber
\end{eqnarray}
Then, for $i=1,...,3$
\begin{eqnarray}
T_{i}:X\ni v\mapsto \left(I\circ
S\right)\left(\frac{\partial}{\partial t}u +
\sum_{j=1}^{3}u_{j}\frac{\partial}{\partial x_{i}}u_{j} -\nu
\sum_{j=1}^{3}\frac{\partial^{2}}{\partial x_{j}^{2}}u_{i}
+\frac{\partial p}{\partial x_{i}}\right)\in Y \label{ExtPDEI}
\end{eqnarray}
and
\begin{eqnarray}
T_{4}: X\ni v\mapsto \left(I\circ
S\right)\left(\sum_{i=1}^{4}\frac{\partial}{\partial
x_{i}}u\right)\in Y\label{ExtPDEII}
\end{eqnarray}
define unique extensions of the components of $T$ to $X$, where
\begin{eqnarray}
X=\mathcal{ML}^{2}_{0}\left(\Omega\right)^{4},\nonumber
\end{eqnarray}
\begin{eqnarray}
Y=\mathcal{ML}^{0}\left(\Omega\right)^{4}\nonumber
\end{eqnarray}
where, for $m\in\mathbb{N}$,
\begin{eqnarray}
\mathcal{ML}^{m}_{0}\left(\Omega\right)=
\left\{u\in\mathcal{ML}^{0}\left(\Omega\right)\begin{array}{|ll}
1) & u\left(\cdot,0\right)\in\mathcal{C}^{m}\left(\mathbb{R}^{3}\right) \\
2) & \begin{array}{ll} \exists & \Gamma\subset
\Omega\mbox{ closed nowhere dense :} \\
 & u\in\mathcal{C}^{m}\left(\Omega\setminus \Gamma\right) \\
\end{array} \\
\end{array}\right\}
\end{eqnarray}
With the initial value problem (\ref{Initval}) we associate the
mapping
\begin{eqnarray}
R_{0}:X\ni u\mapsto u_{|t=0}\in Z\label{Rest}
\end{eqnarray}
where
\begin{eqnarray}
Z=\mathcal{C}^{2}\left(\mathbb{R}^{3},\mathbb{R}^{3}\right)\nonumber
\end{eqnarray}
That is, $R_{0}$ assigns to $u\in X$ the restriction of $u$ to the
hyperplane $\mathbb{R}^{3}\times \{0\}$.  Note that this amounts
to a \textit{separation} of the problem of solving the system of
PDEs (\ref{NSEq}), and the problem of satisfying the initial
value.  This is a characteristic feature of the Order Completion
Method \cite{Obergugenberger and Rosinger}, and the
pseudo-topological version of the theory developed here and in
\cite{van der Walt 3}.  What is more, and as will be seen in the
sequel, this allows for the rather straight forward and easy
treatment of boundary and / or boundary value problems, when
compared to the usual functional analytic methods.

Define the mapping $T_{0}$ as
\begin{eqnarray}
T_{0}:X\ni v=\left(u,p\right)\mapsto \left(Tv,R_{0}u\right)\in
Y\times Z\label{T0Def}
\end{eqnarray}
The mapping $T_{0}$ induces an equivalence relation $\sim_{T_{0}}$
on $X$ through
\begin{eqnarray}
\begin{array}{ll}
\forall & v,w\in X\mbox{ :} \\
& v\sim_{T_{0}}w\Leftrightarrow T_{0}v=T_{0}w \\
\end{array}
\end{eqnarray}
The quotient space $X/\sim_{T_{0}}$ is denotes $X_{T_{0}}$.  There
is then an \textit{injective} mapping
\begin{eqnarray}
\widehat{T}_{0}:X_{T_{0}}\ni V\mapsto\left(T_{0}v,R_{0}u\right)\in
Y\times Z
\end{eqnarray}
where $v=\left(u,p\right)$ is any member of the equivalence class $V$, such that the diagram\\

\begin{math}
\setlength{\unitlength}{1cm} \thicklines
%\begin{pspicture}(13,6)
\begin{picture}(13,6)

\put(2.9,5.4){$X$} \put(3.3,5.5){\vector(1,0){6.0}}
\put(9.5,5.4){$Y\times Z$} \put(5.9,5.7){$T_{0}$}
\put(3.0,5.2){\vector(0,-1){3.5}} \put(3.4,1.4){\vector(1,0){6.0}}
\put(2.9,1.3){$X_{T_{0}}$} \put(9.5,1.3){$Y\times Z$}
\put(2.4,3.4){$q_{T_{0}}$} \put(9.8,3.4){$i$}
\put(9.6,5.2){\vector(0,-1){3.5}} \put(5.9,1.6){$\widehat{T}_{0}$}

\end{picture}
%\end{pspicture}
\end{math}\\
commutes, with $q_{T_{0}}$ the quotient mapping.

We equip the space $\mathcal{ML}^{0}\left(\Omega\right)$ with the
uniform order convergence structure $\mathcal{J}_{o}$, and $Y$
carries the product uniform convergence structure.  The space $Z$
carries the uniform convergence structure $\mathcal{J}_{\lambda}$,
see \cite{Beattie and Butzmann}, associated with the convergence
structure
\begin{eqnarray}
\begin{array}{ll}
\forall & u\in Z\mbox{ :} \\
& \lambda\left(u\right)=[u] \\
\end{array}\label{ZCS}
\end{eqnarray}
That is,
\begin{eqnarray}
\begin{array}{ll}
\forall & \mathcal{U}\mbox{ a filter on $Z\times Z$ :} \\
& \mathcal{U}\in\mathcal{J}_{\lambda}\Leftrightarrow
\left(\begin{array}{ll}
\exists & u_{1},...,u_{k}\in Z\mbox{ :} \\
& \left([u_{1}]\times [u_{1}]\right)\cap...\cap \left([u_{k}]\times [u_{k}]\right)\subseteq \mathcal{U} \\
\end{array}\right) \\
\end{array}
\end{eqnarray}
Note that $\mathcal{J}_{\lambda}$ induces the convergence
structure $\lambda$, and is uniformly Hausdorff and complete
\cite{Beattie and Butzmann}.  In particular, the sequences which
converge with respect to $\mathcal{J}_{\lambda}$ are exactly the
constant sequences. The product space $Y\times Z$ carries the
product uniform convergence structure, which we denote by
$\mathcal{J}_{P}$. In view of Theorem \ref{Completion} and
\cite[Theorem 3.1]{van der Walt 4} the completion $\left(Y\times
Z\right)^{\sharp}$ of $Y\times Z$ is
$\mathcal{NL}\left(\Omega\right)^{4}\times Z$, equipped with the
product uniform convergence structure with respect to the uniform
convergence structure $\mathcal{J}_{o}^{\sharp}$ and the uniform
convergence structure $\mathcal{J}_{\lambda}$.  We equip
$X_{T_{0}}$ with the initial uniform convergence structure
$\mathcal{J}_{T_{0}}$ with respect to the mapping
$\widehat{T}_{0}$.  That is,
\begin{eqnarray}
\begin{array}{ll}
\forall & \mathcal{U}\mbox{ a filter on $X_{T_{0}} \times X_{T_{0}}$ :} \\
& \mathcal{U}\in \mathcal{J}_{T_{0}}\Leftrightarrow \left(\widehat{T}_{0}\times\widehat{T}_{0}\right)\left(\mathcal{U}\right)\in\mathcal{J}_{P} \\
\end{array}
\end{eqnarray}
Since $\widehat{T}_{0}$ is injective, it is a uniformly continuous
embedding so that $X_{T_{0}}$ is uniformly isomorphic to a
subspace of $Y\times Z$.  Therefore, see \cite{van der Walt 4},
the the mapping $\widehat{T}_{0}$ extends to a uniformly
continuous embedding
\begin{eqnarray}
\widehat{T}_{0}^{\sharp}:X_{t_{0}}^{\sharp}\rightarrow
\left(Y\times Z\right)^{\sharp}
\end{eqnarray}
so that $X_{t_{0}}^{\sharp}$ is uniformly isomorphic to a subspace
of $\left(Y\times Z\right)^{\sharp}$.
This is summarized in the following commutative diagram.\\

\begin{math}
\setlength{\unitlength}{1cm} \thicklines
%\begin{pspicture}(13,6)
\begin{picture}(13,6)

\put(2.9,5.4){$X_{T_{0}}$} \put(3.5,5.5){\vector(1,0){5.8}}
\put(9.5,5.4){$Y\times Z$} \put(5.9,5.7){$\widehat{T}_{0}$}
\put(3.0,5.2){\vector(0,-1){3.5}} \put(3.4,1.4){\vector(1,0){6.0}}
\put(2.9,1.3){$X_{T_{0}}^{\sharp}$} \put(9.5,1.3){$\left(Y\times
Z\right)^{\sharp}$} \put(9.6,5.2){\vector(0,-1){3.5}}
\put(5.9,1.6){$\widehat{T}_{0}^{\sharp}$}

\end{picture}
%\end{pspicture}
\end{math}\\
A generalized solution to (\ref{NSEq}) through (\ref{Initval}) is
any $V^{\sharp}\in X_{T_{0}}^{\sharp}$ that satisfies the equation
\begin{eqnarray}
\widehat{T}_{0}^{\sharp}V^{\sharp}=g\label{GenEq}
\end{eqnarray}

The main result of this section, concerning the existence of
generalized solutions to (\ref{NSEq}) through (\ref{Initval}), is
based on the existence of approximate solutions, which follows
form the following \cite{van der Walt 3}.  We include the proof to
illustrate the technique.
\begin{lemma}\label{FirstApprox}
Consider any
$g=\left(f,0\right)\in\mathcal{C}^{0}\left(\Omega\right)$ and any
$\epsilon>0$.  Then
\begin{eqnarray}
\begin{array}{ll}
\forall & \left(x_{0},t_{0}\right)\in\Omega\mbox{ :} \\
\exists & v=\left(u,p\right)\in\mathcal{C}^{2}\left(\Omega\right)\mbox{ :} \\
\exists & \delta>0\mbox{ :} \\
\forall &  \left(x,t\right)\in\Omega\mbox{ :} \\
& \left(\begin{array}{l}
\|x_{0}-x\|<\delta \\
|t_{0}-t|<\delta\\
\end{array}\right)\Rightarrow g\left(x,t\right)-\epsilon< T\left(x,t,D\right)v\left(x,t\right)< g\left(x,t\right) \\
\end{array}
\end{eqnarray}
where the order above is coordinatewise, and $\epsilon$ represents
the $4$ dimensional vector that corresponds to the real number
$\epsilon$.
\end{lemma}
\begin{proof}
Note that, for every $\left(x,t\right)\in\Omega$, the function $F$
satisfies
\begin{eqnarray}
\{F\left(x,t,\xi\right)\mbox{ :
}\xi=\left(\xi_{\alpha}\right)_{|\alpha|\leq2}\in\mathbb{R}^{K}\}=\mathbb{R}^{4}\nonumber
\end{eqnarray}
so that, for every $\left(x,t\right)$ and $\epsilon>0$, there is
some $\xi^{\epsilon}\in\mathbb{R}^{K}$ such that
$F\left(x,t,\xi^{\epsilon}\right)=g\left(x,t\right)$.  Let
$v=\left(u,p\right)$ be the $\mathcal{C}^{2}$-smooth function such
that
\begin{eqnarray}
D^{\alpha}u\left(x,t\right)=\xi^{\epsilon}_{\alpha} \nonumber
\end{eqnarray}
The result now follows from the continuity of $v$, $F$ and $g$.
\end{proof}\\ \\
The following is essentially a version of Lemma \ref{FirstApprox}
above which incorporates the initial condition (\ref{Initval}).
\begin{lemma}\label{SecondApprox}
Let $g$ and $\epsilon$ be as in Lemma \ref{FirstApprox} above.
Consider any $u^{0}\in\mathcal{C}^{2}\left(\mathbb{R}^{3},
\mathbb{R}^{3}\right)$.  Then
\begin{eqnarray}
\begin{array}{ll}
\forall & x_{0}\in\mathbb{R}^{3}\mbox{ :} \\
\exists & v=\left(u,p\right)\in\mathcal{C}^{2}\left(\Omega\right)\mbox{ :} \\
\exists & \delta>0\mbox{ :} \\
& \begin{array}{ll} 1) & \begin{array}{ll}
\forall & \left(x,t\right)\in\Omega\mbox{ :} \\
& \left(\begin{array}{l}
\|x_{0}-x\|<\delta \\
|t|<\delta\\
\end{array}\right)\Rightarrow g\left(x,t\right)-\epsilon< T\left(x,t,D\right)v\left(x,t\right)< g\left(x,t\right) \\
\end{array} \\
2) & x\in\mathbb{R}^{3}\Rightarrow u\left(x,0\right)=u^{0}\left(x\right) \\
\end{array} \\
\end{array}
\end{eqnarray}
\end{lemma}
\begin{proof}
The proof follows similar arguments as those employed in the proof
of Lemma \ref{FirstApprox} when one sets
\begin{eqnarray}
u\left(x,t\right)=u^{0}\left(x\right)+\varphi\left(t\right)\nonumber
\end{eqnarray}
where $\varphi\in\mathcal{C}^{2}\left([0,\infty)\right)$ is an
appropriate function such that $\varphi\left(0\right)=0$.
\end{proof}\\ \\
The main result of this section is now the following.
\begin{theorem}\label{SolEx}
For any $g=\left(f,0\right)\in Y$ and any $u^{0}\in Z$, there
exists a unique $V^{\sharp}\in X_{T_{0}}^{\sharp}$ such that
\begin{eqnarray}
\widehat{T}_{0}^{\sharp}V^{\sharp}=g\label{GE}
\end{eqnarray}
\end{theorem}
\begin{proof}
Let
\begin{equation}\label{AE3}
\Omega=\bigcup_{\nu\in\mathbb{N}}C_{\nu}
\end{equation}
where, for $\nu\in\mathbb{N}$, the compact sets $C_{\nu}$ are
$4$-dimensional intervals
\begin{equation}
C_{\nu}=[a_{\nu},b_{\nu}]
\end{equation}
with $a_{\nu}=\left(a_{\nu,1},...,a_{\nu,n}\right)$,
$b_{\nu}=\left(b_{\nu,1},...,b_{\nu,n}\right)\in\mathbb{R}^{n}$
and $a_{\nu,i}\leq b_{\nu,i}$ for every $i=1,...,n$. We also
assume that $C_{\nu}$, with $\nu\in\mathbb{N}$ are locally finite,
that is,
\begin{equation}\label{AE2}
\begin{array}{ll}
\forall & \left(x,t\right)\in\Omega\mbox{ :} \\
\exists & V_{x}\subseteq\Omega\mbox{ a
neighborhood of }x\mbox{ :} \\
& \{\nu\in\mathbb{N}\mbox{ :
}C_{\nu}\cap V_{x}\neq\emptyset\}\mbox{ is finite} \\
\end{array}
\end{equation}
We also assume that the interiors of $C_{\nu}$, with
$\nu\in\mathbb{N}$, are pairwise disjoint.  We note that
such $C_{\nu}$ exist, see \cite{Forster}.\\
Select $\nu\in\mathbb{N}$ and $\epsilon>0$ arbitrary but fixed.
For any $\left(x,t\right)\in C_{\nu}$, let
$\delta_{\left(x,t\right)}>0$ be the positive number and
$v^{\epsilon}_{\left(x,t\right)}$ the function associated with
$\left(x,t\right)$ through Lemma \ref{FirstApprox}, if $t>0$, and
Lemma \ref{SecondApprox} if $t=0$.  Since $C_{\nu}$ is compact, it
follows that
\begin{eqnarray}
\begin{array}{ll}
\exists & \delta>0\mbox{ :} \\
\forall & \left(x_{0},t_{0}\right)\in C_{\nu}\mbox{ :} \\
\exists & v=\left(u,p\right)\in\mathcal{C}^{2}\left(\mathbb{R}^{4},\mathbb{R}^{4}\right)\mbox{ :} \\
& \begin{array}{ll} 1) &\left(\begin{array}{l}
\|x-x_{0}\|\leq\delta \\
|t-t_{0}|\leq \delta \\
\end{array}\right)\Rightarrow g\left(x,t\right)-\epsilon\leq
T\left(x,t,D\right)v\left(x,t\right)\leq g\left(x,t\right)\mbox{,
}\left(x,t\right)\in\Omega \\
2) & t_{0}=0 \Rightarrow u\left(x\right)=u^{0}\left(x\right)\mbox{, }x\in\mathbb{R}^{3} \\
\end{array} \\
\end{array}\label{AE}
\end{eqnarray}
Subdivide $C_{\nu}$ into $n$-dimensional intervals
$I_{\nu,1},...,I_{\nu,\mu_{\nu}}$ such that their interiors are
pairwise disjoint and
\begin{eqnarray}
\begin{array}{ll}
\forall & \left(x_{0},t_{0}\right)\mbox{, }\left(x,t\right)\in I_{\nu,i}\mbox{ :} \\
& \begin{array}{ll}
1) & \|x_{0}-x\|<\delta \\
2) & |t_{0}-t|<\delta \\
\end{array} \\
\end{array}\nonumber
\end{eqnarray}
If $I_{\nu,i}\cap\left(\mathbb{R}^{3}\times\{0\}\right)
=\emptyset$, take $a_{i}$ to be the center of the interval
$I_{\nu,j}$.  Then by (\ref{AE}) there exists
$v^{\nu,i}=\left(u,p\right)\in\mathcal{C}^{2}\left(\mathbb{R}^{4}\times
\mathbb{R}^{4}\right)$ such that
\begin{equation}\label{AE1}
g\left(x,t\right)-\epsilon\leq
T\left(x,t,D\right)v^{\nu,i}\left(x,t\right)\leq
g\left(x,t\right)\mbox{, }\left(x,t\right)\in I_{\nu,i}
\end{equation}
If, on the other hand,
$I_{\nu,i}\cap\left(\mathbb{R}^{3}\times\{0\}\right)
\neq\emptyset$, let $a_{i}$ denote the projection of the midpoint
of $I_{\nu,i}$ on the hyperplane $\mathbb{R}^{3}\times\{0\}$. Then
by (\ref{AE}) there exists
$v^{\nu,i}=\left(u,p\right)\in\mathcal{C}^{2}\left(\mathbb{R}^{4}\times
\mathbb{R}^{4}\right)$ such that (\ref{AE1}) holds and
\begin{eqnarray}
u\left(x,0\right)=u^{0}\left(x\right)\mbox{, }\left(x,0\right)\in
\left(\mathbb{R}^{3}\times\{0\}\right)\cap I_{\nu,i}\label{AE3}
\end{eqnarray}
Now set
\begin{eqnarray}
v^{\epsilon}=\left(u^{\epsilon}_{1},u^{\epsilon}_{2},u^{\epsilon}_{3},p^{\epsilon}\right)
=\sum_{\nu\in\mathbb{N}}\left(\sum_{i=1}^{\mu_{\nu}}v^{\nu,i}\chi_{I_{nu,i}}\right)
\end{eqnarray}
where $\chi_{I_{\nu,i}}$ is the characteristic function of
$I_{\nu,i}$.  Clearly, $v^{\epsilon}=
\left(u^{\epsilon},p^{\epsilon}\right)$ is
$\mathcal{C}^{2}$-smooth everywhere except on a closed nowhere
dense set, which has measure $0$, and
$u^{\epsilon}\left(x,0\right)=u^{0}\left(x\right)$ everywhere
except on a closed nowhere dense subset of
$\mathbb{R}^{3}\times\{0\}$.\\
Now set $w^{\epsilon}=\left(u^{\epsilon *}_{1},u^{\epsilon *}_{2},
u^{\epsilon *}_{3},p^{\epsilon *}\right)$ where, for $j=1,...,3$
\begin{eqnarray}
u^{\epsilon *}_{j}=\left(I\circ
S\right)\left(u^{\epsilon}_{j}\right)\nonumber
\end{eqnarray}
and
\begin{eqnarray}
p^{\epsilon *}= \left(I\circ S\right)\left(p^{\epsilon}\right)
\end{eqnarray}
Clearly the function $w^{\epsilon}$ belongs to $X$.  What is
more,in view of (\ref{AE1}) through (\ref{AE3}), it follows that
\begin{eqnarray}
g-\epsilon\leq Tw^{\epsilon}\leq g\nonumber
\end{eqnarray}
and
\begin{eqnarray}
R_{0}w^{\epsilon}=u^{0}\nonumber
\end{eqnarray}
so that the sequence
$\left(T_{0}w_{n}\right)=\left(T_{0}w^{\frac{1}{n}}\right)$
converges to $\left(g,u^{0}\right)$ in $Y\times Z$.  For each
$n\in\mathbb{N}$, let $W_{n}$ denote the
$\sim_{T_{0}}$-equivalence class generated by the function
$w^{\frac{1}{n}}$.  The sequence $\left(W_{n}\right)$ is Cauchy in
$X_{T_{0}}$, and since $\widehat{T}_{0}$ is uniformly continuous,
there exists $V^{\sharp}\in X_{T_{0}}^{\sharp}$ that satisfies
(\ref{GE}). Moreover, $V^{\sharp}$ is unique, since the mapping
$\widehat{T}_{0}^{\sharp}$ is a uniformly continuous embedding.
\end{proof}\\ \\
The uniqueness of the generalized solution should not be
misinterpreted.  Note that the completion of $X_{T_{0}}$ consists
of equivalence classes of Cauchy filters on $X_{T_{0}}$, under the
equivalence relation
\begin{eqnarray}
\mathcal{F}\sim_{C}\mathcal{G}\Leftrightarrow
\left(\begin{array}{ll}
\exists & \mathcal{H}\mbox{ a Cauchy filter :} \\
& \mathcal{H}\subseteq \mathcal{F}\cap\mathcal{G} \\
\end{array}\right)\nonumber
\end{eqnarray}
In view of this, the solution $V^{\sharp}$ is actually the
equivalence class of filters $\mathcal{F}$ on $X_{T_{0}}$ such
that $\widehat{T}_{0}\left(\mathcal{F}\right)$ converges to
$\left(g,u^{0}\right)$ in $Y\times Z$.  What is more, $V^{\sharp}$
contains also all \textit{classical}, or smooth, solutions to
(\ref{NSEq}) through (\ref{Initval}), as well as all
\textit{nonclassical} solutions $v=\left(u,p\right)\in
\mathcal{C}^{2}_{nd}\left(\Omega\right)^{4}$, since each such a
solution generates a Cauchy sequence in $X_{T_{0}}$. Hence our
notion of a generalized solution is \textit{consistent} with the
usual classical \textit{and} nonclassical solutions in
$\mathcal{C}_{nd}^{2}\left(\Omega\right)^{4}$ to (\ref{NSEq})
through (\ref{Initval}).  Note that the method presented here for
the three dimensional Navier-Stokes equations applies equally well
to any dimension $n\geq 2$.

\section{Conclusion}

We have constructed an order isomorphic representation
$\mathcal{ML}\left(X\right)$ of the quotient space
$\mathcal{M}\left(X\right)$ consisting of normal lower
semi-continuous functions on $X$.  A nontrivial uniform
convergence structure on $\mathcal{ML}\left(X\right)$, which
induces the order convergence structure was constructed solely in
terms of the order on $\mathcal{ML}\left(X\right)$. The completion
of the uniform convergence space $\mathcal{ML}\left(\Omega\right)$
is obtained as the set $\mathcal{NL}\left(X\right)$ of nearly
finite normal lower semi-continuous functions on $X$.  This result
essentially relies on the fact that $\mathcal{NL}\left(X\right)$
is the Dedekind completion of $\mathcal{ML}\left(X\right)$.  Hence
we have established a topological type model for the Dedekind
completion of the space $\mathcal{ML}\left(X\right)$.  This
includes the case when $X=\Omega$ is a subset of $\mathbb{R}^{n}$,
which is relevant to PDEs.  This makes it possible to enrich the
Order Completion Method for arbitrary nonlinear PDEs of the form
(\ref{PDE}), by reformulating it within the framework of uniform
convergence spaces.  In this regard, we obtained the existence of
generalized solutions to the Navier-Stokes equations in three
spatial dimensions, subject to an initial condition.

\end{document}